\documentclass[11pt,reqno]{amsart}
\usepackage{amssymb,amsmath}
\usepackage[bbgreekl]{mathbbol}
\usepackage[toc,page]{appendix}
\usepackage{mathrsfs}
\usepackage[all]{xy}
\newcommand{\aut}{{\rm Aut}}

\newcommand{\autpol}{{\rm Aut_\K(\pol)}}

\newcommand{\jon}{{\rm J(\pol)}}

\newcommand{\deig}{\rm {\mathbb e}}

\renewcommand{\P}{{\mathbb P}}

\newcommand{\A}{{\mathbb A}}

\newcommand{\C}{{\mathbb C}}

\newcommand{\F}{{\mathbb F}}

\newcommand{\K}{{\Bbbk}}
\newcommand{\N}{{\mathbb N}}

\newcommand{\Q}{{\mathbb Q}}
\newcommand{\Z}{{\mathbb Z}}

\newcommand{\pol}{{\K[x,y]}}

\newcommand{\derpol}{{\rm Der}_\K(\pol)}

\newtheorem{thm}{Theorem}[section]
\newtheorem{pro}[thm]{Proposition}

\newtheorem{lem}[thm]{Lemma}

\theoremstyle{remark}
\newtheorem{rem}[thm]{Remark}

\newtheorem{exa}[thm]{Example}

\begin{document}

\title[Dimension two polynomial locally nilpotent derivations]{A characterization of local nilpotence for dimension two polynomial derivations}
\maketitle
\begin{center}
  {\sc Ivan Pan}\footnote{Research of I. Pan was partially supported by ANII and PEDECIBA, of Uruguay.}
\end{center}

\begin{abstract}
Let $\K$ be an algebraically closed field. We prove that a polynomial $\K$-derivation $D$ in two variables is locally nilpotent if and only if the subgroup of polynomial $\K$-automorphisms which commute with $D$ admits elements whose degree is arbitrary big.   
\end{abstract}
\section{introduction}\label{sec_intro}
Let $\K$ be a field of characteristic 0. A well known result of Rentschler (see \cite{Re}) says that a polynomial derivation $D:\pol\to\pol$ (over $\K$) is locally nilpotent, i.e. for any polynomial $f$ one has $D^nf=0$ for some $n\geq 1$, if and only if $D$ is conjugate to a derivation of the form $u(x)\partial_y$, by means of a suitable polynomial automorphism $\varphi\in\autpol$. If $\aut(D)$ denotes the isotropy subgroup of $D$ with respect to the natural conjugation action of $\autpol$ on the set $\derpol$ of $\K$ derivations, we deduce that $D$ being locally nilpotent implies $\aut(D)$ is conjugate to a subgroup of the form  
\[J_u=\{(\alpha x+\beta,\gamma y+P(x); \alpha,\gamma\in\K^*,\beta\in\K, P\in\K[x], u(\alpha x+\beta)/u(x)=\gamma\},\]
for some $u\in\K[x]$; here we write $\varphi=(f,g)\in\autpol$ to mean $\varphi(x)=f$ and $\varphi(y)=g$.

On the other hand, we consider on $\autpol$ the so-called \emph{inductive topology} defined by the filtration $\A_1\subset \A_2\subset\cdots\subset \A_d\subset\cdots$, where $\A_d=\{(f,g); \deg f,\deg g\leq d\}$. If $\K$ is algebraically closed, we see that $\A_i$ is an affine variety for any $i$ and the subgroup above is an \emph{infinite dimension algebraic group} in the sense of \cite{Sha} or \cite{Ka}. Following this last reference, we conclude that $\aut(u(x)\partial_y)$ doesn't satisfy the property of acting algebraically on $\pol$ as an (usual) algebraic group. Indeed, as shown there for a subgroup of $\autpol$ that property is equivalent to being closed (which $\aut(D)$ does for any derivation $D$: see \cite[Cor. 2.2]{BP2}) and having bounded degree.   

The aim of the present note is to prove that that holds true exclusively for locally nilpotent derivations. More precisely, we have the following characterization of the local nilpotence property whose proof relies strongly on the results of \cite{BP2}. 

\begin{thm}\label{thm1}
  Let $D$ be a nonzero derivation of $\K[x,y]$ where $\K$ is a field algebraically closed of characteristic 0. Then the following assertions are equivalent:

  $(a)$ $D$ is locally nilpotent;

  $(b)$ for every $d$ there exists $(f,g)\in\aut(D)$ such that $\deg f\geq d$ or $\deg g\geq d$; 

  $(c)$ $\aut(D)$ is not an algebraic subgroup of $\aut(\K[x,y])$;

  $(d)$ $\aut(D)$ is an infinite dimensional algebraic subgroup of $\aut(\K[x,y])$. 
\end{thm}

Note that whereas $(b)$ and $(c)$ are equivalent and both are a consequence of $(d)$ the converse is not, \emph{a priori}, necessarily true because $\Z$ is not an algebraic subgroup of $\K$. Moreover, since conjugating is a homeomorphism which respects the degree's boundedness, then from Rentchler's result we deduce that $(a)$ implies $(d)$. Thus we only need to prove that $(b)$ implies $(a)$.

Notice also that Theorem \ref{thm1} says that $\aut(D)$ is an algebraic group when $D$ is not locally nilpotent and then one may ask what kind of such a group corresponds to the conjugation class of a non locally nilpotent derivation.

Finally, we observe that Theorem \ref{thm1} doesn't hold true for dimension 3 or higher (Example \ref{exa1}) even though assertion $(b)$ does for any locally nilpotent derivation $D$ (\cite[Remark 3]{BV}).

\begin{rem}
What we observe here about $\aut(D)$ recalls what happens with the isotropy of another action of $\autpol$. Indeed, it makes one think of the natural action of that group on the set of reduced principal ideals of height 1 (i.e. algebraic plane curves) where the isotropy of such an ideal $(f)$ is not an algebraic group if and only if its generator $f\in\pol$ may be transformed into an element in $\K[x]$ by means of an automorphism (see \cite{BlSt} and references therein). In other words, the ideal generated by such an $f$ would correspond in our research to a locally nilpotent derivation. Moreover, when $n>2$ the phenomena observed in Example \ref{exa1} matches the one described in \cite[Ex. 3.16]{BlSt}.
  \end{rem}
\section{The proof}

  We denote $\deig(D)$ the number of $D$-stable reduced principal ideals of height 1. If $(h)$ is such a principal $D$-stable ideal, then $D(h)=\lambda h$ for some $\lambda\in\pol$ and we say $h$ is a \emph{eigenvector} of $D$  and $\lambda$ is its (corresponding) \emph{eigenvalue}.  In the case where $h$ is an eigenvector which is reduced, i.e. square-free, we will also say that $D$ stabilizes the curve of equation $h=0$. Notice that $\deig(D)=0$ implies that the \emph{kernel} $\ker D$ of $D$ is equal to $\K$.

  Two elements $f_1,f_2\in\pol$ are said to be \emph{equivalent} if there is $\varphi\in\autpol$ such that $\varphi(f_1)=f_2$. In the case where $f_1$ is equivalent to $x$ we say it is \emph{rectifiable}.
  
 We keep all notations introduced in the precedent section. If $\varphi=(f,g)\in\autpol$, we denote $\deg\varphi$ the greatest degree of $f$ and $g$ and call it the \emph{degree} of $\varphi$; notice that we have a degree function $\deg:\autpol\to\N$ which verifies $\deg \varphi\psi\leq\deg\varphi\deg\psi$.

\begin{lem}\label{lem1}
  Let $D\in\derpol$ be a nonzero derivation and assume $\aut(D)$ to be a non-algebraic group. Then one of the following assertions holds

  $(a)$ $0<\deig(D)<\infty$, all irreducible eigenvector is rectifiable and $D$ is conjugate to a derivation of the form $a\partial_x+b\partial_y$ where either $ab\neq 0$ and $x\in\pol$ divides $a$ or $a=0$ and $b\in\K[x]$.

  $(b)$ $\deig(D)=\infty$ and $D$ is conjugate to a derivation of the form $b(x)\partial_y$.

  $(c)$ $\deig(D)=\infty$, $\ker(D)=\K$ and $D$ stabilizes the members of a pencil of rational curves.

  $(d)$ $\deig(D)=0$. 
\end{lem}

\begin{proof}
  We assume $\deig(D)\neq 0$ and prove that one of the assertions $(a), (b)$ or $(c)$ holds.

  First suppose $0<\deig(D)<\infty$. By \cite[Thm. A]{BP2} all irreducible eigenvector is rectifiable. Then there is $\varphi\in\autpol$ such that $\varphi D \varphi^{-1}=a\partial_x+b\partial_y$ admits $x\in\pol$ as an eigenvector. Hence $x$ divides $a$. From \cite[Thm. B]{BP2} we deduce the assertion $(a)$ holds in this case.

  Next suppose $\deig(D)=\infty$. If $\ker D\neq\K$, the references already cited imply we are in the situation of assertion $(b)$. Analogously, if  $\ker D=\K$, then \cite[Thms. D]{BP2} implies we are as in assertion $(c)$ which completes the proof.  
  \end{proof}

  We consider the compactifications $X=\K^2\cup B$ of $\K^2$, where $B$ is the union of at most two curves isomorphic to $\P^1$: either  $X=\F_n, n\geq 1$, where $\F_n$ is the $n^{th}$ Nagata-Hirzebruch surface and $B$ is the union of a fiber and the $(-n)$-curve in that surface or  $X=\P^2$ is the projective plane, with  $B=L_\infty$ the line at infinity with respect to the affine chart $\K^2$. We denote by $\aut(X,B)$ the group of automorphisms of $X$ which leave $B$ invariant. 
  
  \begin{pro}\label{pro1}
    Let $D\in\derpol$ be a nonzero derivation and assume $\aut(D)$ to be a non-algebraic group.
If $\deig(D)\neq 0$, then $D$ is locally nilpotent.
\end{pro}
\begin{proof}
  By Lemma \ref{lem1} we know $D$ satisfies $(a)$, $(b)$ or $(c)$ therein.  Moreover, in case (b) the assertion is obvious. Let us consider the case (c), and denote $\Lambda$ a pencil of rational curves whose members are stable under $D$. From \cite[Pro. 2.10, Cor. 2.12]{BP2} we deduce that up to conjugation $\aut(D)$ may be thought of as either a subgroup of $\aut(\F_n,B)$, for some $n$, or one of $\aut(\P^2,L_\infty)$, where in the second case $\Lambda$ turns out to be a pencil composed by lines passing through a point $p\in L_\infty$. Without loss of generality we assume $\K^2\subset\P^2$ via the embedding $(x,y)\mapsto (1:x:y)$, with $L_\infty=(x_0=0)$ and $p=(0:0:1)$.

Since $\aut(\F_n,B)$ is an algebraic group then $\aut(D)$ is necessarily as in the second case, hence $\aut(D)$ is contained in the so-called \emph{de Jonqui\`eres Group}
  \[\jon:= \{(\alpha x+\beta,\gamma y+P(x); \alpha,\beta\in\K^*,\gamma\in\K, P\in\K[x]\}.\]
  Moreover, since a general member of  $\Lambda$ corresponds in $\K^2$ to a line of equation $x-\beta=0$, for a general $\beta\in\K$, we conclude that if $D=a\partial_x+b\partial_y$, then  $x-\beta$ divides $D(x-\beta)=a$, for any $\beta$. Hence $a=0$ and $D=b(x,y)\partial_y$. Thus the assertion is consequence of \cite[Thm. B]{BP2}.

  Now, assume we are as in the assertion (a) of Lemma \ref{lem1}. Up to conjugation we may assume $D=x^\ell a(x,y)\partial_x+b(x,y)\partial_y$, where we may suppose $a\neq 0$ and $x$ does not divide $a$ because $a=0$ and $b\in\K[x]$ leads to the required conclusion.      

We have a homomorphism $\aut(D)\to {\rm Per(E)}$, where $E$ is the set of prime principal ideals of $\pol$ which are $D$-stable and ${\rm Per(E)}$ denotes the finite group of permutations of $E$. Hence the principal ideal $x\pol$ belongs to $E$. Since the degree function $\deg:\autpol\to\N$ is not bounded on $\aut(D)$ we deduce it is not bounded on the kernel $K$ of that homomorphism, so $K$ is not an algebraic group. Note that an element $\varphi=(f,g)\in K$ verifies $\varphi(x)/x\in\K^*$, hence $f=\alpha x$ for some $\alpha\in\K^*$. Since the jacobian of $\varphi$ is constant we deduce $g=\gamma y+P(x)$ for some $\gamma\in\K^*$ and $P\in\K[x]$. In other words $K$ is contained in the subgroup of $\jon$ whose elements fix the ideal generated by $x$. More explicitly, if  $\varphi=(\alpha x,\gamma y+P(x))\in K$, then
  \begin{equation}\label{eq1}
    \left\{\begin{array}{l}
      \alpha^{\ell-1}a(\alpha x,\gamma y+P(x))=a(x,y)\\
          b(\alpha x,\gamma y+P(x))=\gamma b(x,y)+x^\ell a(x,y)P'(x).
      \end{array}\right.\end{equation}

  On the other hand, since $K$ is not an algebraic group we deduce it contains a sequence $(\varphi_n)_{n\geq 1}$ of elements such that the corresponding sequence of degrees $(\deg\varphi_n)_{n\geq 1}$ is increasing.  We will show in several steps this implies $a=0$ which yields a contradiction and terminates the proof. Write $\varphi_n=(\alpha_nx,\gamma_n y+P_n(x))$, $n\geq 1$.

  First we observe that $a$ does not depend on $y$. Indeed, write $a=\sum_{i=0}^d a_i(x)y^i$, $d\geq 0$. Since $\deg P_n$ increases with $n$, if $d>0$ we see that the top equality in (\ref{eq1}) implies $\deg P_n(x)^d$ is bounded,  hence $d=0$.

  Second, by an analogous reasoning the bottom equality in (\ref{eq1}) gives $b=b_0(x)+b_1(x)y$, where $b_1\neq 0$, and then that equality is equivalent to the following two ones 
  \begin{eqnarray}
    b_0(\alpha x)+b_1(\alpha x)P(x)&=&\gamma b_0(x)+x^\ell a(x)P'(x),\label{eq2}\\
    b_1(\alpha x)&=&b_1(x).\label{eq3}
  \end{eqnarray}
  
  Now write
  \[a=\sum_{i=0}^rA_i x^i,\ b_1=\sum_{i=0}^sB_i x^i,\]
  where $A_rB_s\neq 0$. If $P=P_n=\sum_{i=0}^mp_i x^i$ for some $n\gg 0$, we deduce
 \[m=m_n=\alpha_n^sB_sA_r^{-1}.\]
  Hence we may suppose $\alpha_n^s\neq 1$ because $m_n$ increases with $n$, and so $s=\deg b_1>0$. From (\ref{eq3}) we deduce $\alpha=1$, a contradiction which finishes the proof.    
\end{proof}

Now we treat the case $(d)$ of Lemma \ref{lem1}. We have the following result valid over an arbitrary field of characteristic zero which together with Proposition \ref{pro1} readily leads Theorem \ref{thm1}.

\begin{pro}
Let $D$ be a derivation such that $\deig(D)=0$. Then $\aut(D)$ is finite.
 \end{pro}
 \begin{proof}
   If $\K=\C$ is the field of complex numbers, the result is a straightforward consequence of \cite[Thm. A]{CMP}. Denote by $\ell$ the order of $\aut(D)$ in that case. It suffices to prove that if $F=\{\varphi_1,\ldots,\varphi_n\}$ is a subset of $\aut(D)$ in the general case, then $n\leq \ell$.

 In fact, inspired by the proof of \cite[Prop. 1.4]{DK} we consider the extension $\K_0$ of $\Q$ obtained by adjoining the coefficients of $D(x), D(y)$, $\varphi_i(x)$, $\varphi_i(y)$, $i=\ldots,n$. Then $D$ and all of the $\varphi_i'$s restraint to give a derivation and suitable automorphisms $D_0, \varphi_{i0}:\K_0[x,y]\to\K_0[x,y]$, $i=1\ldots,n$, such that $\varphi_{i0}D=D\varphi_{i0}$ for all $i$. Since $\K_0$ is isomorphic to a subfield of $\C$ we may suppose $\K_0\subset \C$, and then all these maps extend to $\C[x,y]$ from which the assertion follows.
\end{proof}

\begin{exa}\label{exa1}
  Theorem \ref{thm1} doesn't hold true for $n>2$. Indeed, let $D=\sum_{i=1}^{n}a_i\partial_{x_i}$ be a derivation of $B=\K[x_1,\ldots,x_n]$, where $a_1,\ldots,a_{n-1}$ don't depend on $x_1, x_n$ and $a_n=0$. Then $D$ induces a derivation in $A=\K[x_2,\ldots,x_{n-1}]$. If $a_2,\ldots,a_{n-1}$ are general enough to ensure $D$ is not locally nilpotent as derivation in $A$, then it is so also as derivation in $B$. However, $\aut(D)$ contains the automorphisms of the form
  \[(x_1+p(x_{n}),\ldots,x_{n-1},x_n), \ p\in \K[x_n],\]
  hence it contains elements defined by polynomials of arbitrary degree.
\end{exa}


\end{document}